\documentclass[titlepage,draft,11pt]{article} 
\usepackage{amsfonts,latexsym} 
\usepackage{pstricks,pst-plot}
\usepackage{amsmath} 
\usepackage{amsthm}
\usepackage[a4paper]{geometry}

\date{}

\newtheorem{thm}{Theorem}[section]

\newtheorem{rmk}[thm]{Remark}
\newtheorem{prop}[thm]{Proposition}

\title{ A simple toy model for the collapse of the local group with the Shapley attractor.}

\author{Alain Haraux\vspace{1ex}\\ 
{\normalsize Sorbonne Universit\'e, Universit\'e Paris-Diderot SPC, CNRS, INRIA}, \\
{\normalsize Laboratoire Jacques-Louis Lions,  LJLL, F-75005,
Paris, France.}\\ 
{\normalsize e-mail: \texttt{alain.haraux@sorbonne-universite.fr}}}

\begin{document}
\maketitle
\begin{abstract}
A toy model is proposed for the Cosmic Dipole  consisting in the Shapley attractor and the so-called Dipole repeller, whose action is assimilated to an anti-gravitational force. According to  this model, the local group will collapse in finite time with the Shapley attractor and, by using the available figures for distances and masses, it is shown that it will happen in less than 100 billion years. To obtain a more precise estimate, more knowledge will be necessary on the equivalent negative mass of the repeller. \\

\vspace{1cm} 

 \end{abstract}

 
\section{Introduction}
Since the discovery in 2017 of the great repeller \cite{Hoff}, it is currently admitted that our local group of galaxies is trapped between the Shapley attractor discovered in 2006 \cite{Proust} and a relatively empty zone playing the role of a repeller. Concerning the repeller, there is a controversy between 2 interpretations, the dominant idea being that the relative void repels because its density is lower than the average density in the region, and some other specialists invoking a negative repelling mass to explain the presence of the local vacuum. In both cases, in the equation, the repeller can be modelized by a negative mass, virtual in the first case and real in the second. To decide whether  the negative mass is real or virtual, we would need to know whether the repulsive force is stronger or not than the one created by an absolute vacuum in the volume of the repeller, and unfortunately this may be difficult, requiring to observe very precisely the motion of objects close to the repeller. \\

In this short note, we introduce a very simplified toy model to understand mathematically the convergence to the attractor. For this we reduce both attractor and repeller to points and we consider them as fixed in space, since their mass is considerable compared to that of galaxies. So the convergence is towards the attractor and not some kind of barycenter, which is always the case in reality for 2 body problems. We also neglect the motion of the repeller with respect to the attractor, while the exact nature of the repeller is unknown. If the repeller appears in the future to have a real strongly negative mass comparable in absolute value to that of the Shapley attractor, our conclusions will have to be changed since in that case the repeller and the attractor will repell each other and this may change the global behavior. \\

The plan of the paper is as follows: Sections 2 and 3 are devoted to explicit calculations without repeller. In Sections 4 , we show that the presence of the repeller produces an earlier collapse. Section 5 and 6 are devoted to the study of the repeller. In Section 7, we briefly discuss the significance of the anti-gravitational action of the repeller. Section 8 is an appendix collecting the numerical data used in the rest of the paper. 

\section{Some calculations without the repeller} 
In this section, we evaluate the collapsing time for a 2-body problem, assuming that  the largest body is so massive that we can consider it fixed. Denoting by $M$ its mass, and assuming that the motion takes place on a line passing through it, with initial velocity pointing towards the attractor put at the origin,  the equation of variation of the distance $ w(t)$ of the light body to the attractor is scalar and of the form 
$$ w''= - \frac{GM}{w^2}, \quad w(0) = W>0, \,\,w'(0) <0.  $$ The energy $$ E := \frac{1}{2}w'^2 - \frac{GM}{w} = \frac{1}{2}w'^2(0) - \frac{GM}{W} $$ is conserved, which implies that $w$ satisfies the first order ODE: 
$$ w'= - \sqrt{ \frac{K}{w} + C }$$ where $$ K= 2GM, \quad C = w'^2(0) - \frac{2GM}{W}. $$ 

The  behavior of w(t) will of course depend on the sign of the constant $C$, although near the collapse the term $\displaystyle \frac{K}{w}$ is clearly predominant. To fix the ideas, let us see how the values of the two terms compare in the case of the Shapley attractor and the local group. A rough estimate corresponding to the data found on the web gives (see appendix) $$ Ww'^2(0)\sim 2. 10^{36} m^3/ s^2; \quad K = 2GM\sim  2.10^{37} m^3/ s^2 $$ so that clearly $C<0$. We shall do the next calculations under this hypothesis, so that the equation becomes 
\begin{equation} \label{1} w'= - \sqrt{ \frac{K}{w} -\beta }\end{equation} with 
\begin{equation} \label{neg} \beta = -C = \frac{K}{W}- w'^2(0)>0 \end{equation} It is clear that \eqref{1} has a unique local solution with initial condition $w(0) = W$ on a small positive interval of time $J = (0, \tau).$ Since $w$ is decreasing , we can see that the solution can be continued until  $w(t)$ vanishes. And because $w'(t) < w'(0)$ on the existence interval, we find that the solution vanishes in a finite time $T$ with $$ 0 = w(T) = W + \int_0^T w'(s) ds \le W+ T w'(0)  $$  hence $$T
 \le \frac{W}{|w'(0)|}.$$ In the example we have $$ \frac{W}{|w'(0)|}\sim 10^{19} s \sim 3. 10^{11} \rm{years}. $$ As we shall see, an exact calculation will allow to reduce that number, which is natural since the velocity increases in absolute value and  even becomes infinite at the collapsing time. 
  
 \section{Exact calculation of the collapse time without repeller} Equation \eqref{1} gives easily the exact value of $T$. Indeed, it can be written in the form 
 \begin{equation*}  \frac{dt}{dw } = - \frac{1} {\sqrt{ \frac{K}{w} -\beta }} = - \frac{\sqrt{w}} {\sqrt{ K -\beta w}} \end{equation*} yieding 
 \begin{equation} T = \int_0^W  \frac{w^{1/2} dw} {\sqrt{ K -\beta w}}\end{equation} 
 We can now state 
  \begin{prop} \label{Prop1}Under condition \eqref{neg}, we have \begin{equation} \label {T} T = \frac{1}{K^{1/2} \gamma^{3/2} } \left [ \arg \sin \sqrt {\gamma W} - \sqrt {\gamma W} \sqrt {(1-\gamma W)}\right] \end{equation} with $$ \gamma = \frac{\beta}{K}. $$ \end{prop}

  \begin{proof} To compute this integral, we set $$ s = \gamma w^{1/2} \Longleftrightarrow w = \frac{s^2}{\gamma} $$ yielding the formula
  \begin{equation*}  T = \frac{2}{K^{1/2} \gamma^{3/2} } \int_0^{({\gamma W})^{1/2}}   \frac{s^2 ds} {\sqrt{ 1 -s^2}} \end{equation*} The calculation is classical by the change of variable 
  $ s = \sin \phi $ and we obtain \eqref{T}.\end{proof} 
  
  \begin{rmk} Note that this formula makes sense since $$ \arg \sin \sqrt {\gamma W} >\sqrt {\gamma W} ; \quad \gamma W= \frac{\beta W}{K} = 1- \frac{Ww'^2(0)}{K}  \in (0, 1). $$ \end{rmk}
   \noindent Actually $\gamma W \sim 1$, so we get the first approximation 
  \begin{equation} T \sim \frac{\pi}{2} \frac {W^{3/2}}{\sqrt{2GM} }  \end{equation} valid under the assumption  $ Ww'^2(0) <<1$. This formula makes sense from the point of view of dimensional analysis, since $[G] = M^{-1} L^3T^{-2}$. We can refine it a little bit since in our precise numerical situation, we have $\gamma W\sim 0.9$, hence $\frac{\pi}{2}$ can be replaced by $$ \arg \sin \sqrt {\gamma W} - \sqrt {\gamma W} \sqrt {(1-\gamma W)}\sim \arg \sin 0.95 - 0.95 \sqrt {0.1}\sim 1,253- 0,3\sim 0, 95 $$ We end up with 
  \begin{equation} \label {Tatt} T \sim 3.076. 10^{18} s \sim 10^{11} \rm{years}  \end{equation} So with the attractor alone, the collapse would take place in 100 billion years. The exact calculation improved the estimate by a factor 3 with respect to the previous section.

\section{The equation including the repeller}  We can imagine that if the effect of the repeller is roughly equivalent to that of the attractor, the collapsing time will be significantly reduced. However it would not be divided by 2, since when approaching the attractor, the influence of the repeller will be roughly divided by 4 as a consequence of  the inverse square law. From the theoretical point of view, it is nevertheless interesting to see how the equations look like. Modeling the attractor by a positive mass $M$ located at $0$ and the repeller by a (virtual or real) negative mass denoted by  $ -M_*, M_*\ge 0$ located on the positive axis with abscissa $L>0$, the equation now becomes 

$$ u''= - \frac{GM}{u^2}- \frac{GM_*}{(L-u)^2}, \quad u(0) = U\in (0,L), \,\,u'(0) <0. $$ We set $ 2GM = a; \quad 2GM_* = b $ so that the equation becomes 

$$ u''= - \frac{a}{2u^2}- \frac{b}{2(L-u)^2} $$ 

The energy 
$$ F := \frac{1}{2}\left[u'^2 - \frac{a}{u} + \frac{b}{L-u}\right]= \frac{1}{2}\left[u'^2(0) - \frac{a}{U} + \frac{b}{L- U}\right] $$ is conserved, which implies that $u$ now satisfies the first order ODE: 
\begin{equation} \label{rep} u'= - \sqrt{ \frac{a}{u} - \frac{b}{L-u}+ C }\end{equation} where $$  C = u'^2(0) - \frac{a}{U} + \frac{b}{L- U}. $$ 

As for equation \eqref{1}, it is clear that \eqref{rep} has a unique local solution with initial condition $u(0) = U$ on a small positive interval of time $J = (0, \tau).$ Since $u$ is decreasing as long as it exists, we can see that the solution can be continued until  $u(t)$ vanishes. And because $u'(t) < u'(0)$ on the existence interval, we find, as in the case of  \eqref{1}, that the solution vanishes in a finite time $T$ with  $$T\le \frac{U}{|u'(0)|}.$$  In addition we have the following result confirming the expected property that the addition of a repeller reduces the collapsing time: 
\begin{prop} \label{comp} Assuming $$ u'^2(0) - \frac{a}{U}<0, $$  the collapsing time for \eqref{rep} is less than the value given by Proposition \ref{Prop1} with $$ W= U; \quad  \beta = \frac{a}{U}- u'^2(0). $$  \end{prop} 
\begin{proof} This comes from the fact that along the trajectory we have constantly $u(t) \le w(t)$ because $u(0)= w(0)$ and $u$ is a strict sub-solution of \eqref{1}. \end{proof} 
\begin{rmk} We have an integral formula for the collapsing time which can easily be computed as 
\begin{equation} \label {Tcomp} T = \int _0^U \frac{\sqrt{u(L-u)} du}{ \sqrt{Cu(L-u)-(a+b)u + aL}} \end{equation} It is not difficult, using this formula, to recover the result of Proposition \ref{comp} since 
$$Cu(L-u) -bu > C u (L-U) -bu = -\beta u (L-u)  $$  for all $u\in (0, L).$ Then we find $$ T <  \int _0^U \frac{\sqrt{u(L-u)} du}{\sqrt {-au + aL -\beta u (L-u) }} =  \int _0^U \frac{\sqrt{u} du}{\sqrt {a -\beta u }}. $$ \end{rmk} 
\begin{rmk} We do not know if the integral \eqref{Tcomp} is computable in terms of elementary functions. A  numerical integration with the numerical data of the appendix and the choice $M^*= M$ gives the value  \begin{equation} \label{tot} T\sim 2,11. 10^{18} s \sim 6,7. 10^{10} years \end{equation}  about 2 thirds of the value without repeller. \end{rmk} 
\begin{rmk} Due to the very large value of T, we may expect that a lot of things happen before the collapse. Some galaxies will collide, for instance the Milky way and the Andromeda galaxy M31, in a much shorter time, at least one order of magnitude smaller. The structure of the Shapley attractor itself will evolve since some other galaxies, closer to the attractor, will collapse much before the local group. And finally it is possible that a ``big crunch" takes place much sooner, or alternatively that the expansion of space increases the size of the Dipole in such a way that the collapse never happens. \end{rmk} 

\section{Exact calculation of the collapse time with the repeller alone} It is interesting to evaluate the collapse time when $a = 0$ (repeller alone). In this case \eqref {Tcomp} reduces to \begin{equation} \label {Trep} T = \int _0^U \frac{\sqrt{(L-u)} du}{ \sqrt{C(L-u)-b}} \end{equation} with 
$$  C = u'^2(0)  + \frac{b}{L- U}>0. $$  By making the change of variables $x= \sqrt{\frac{C}{b}}\sqrt{L-u} $, we find $$  T = \frac{2b}{C\sqrt{C}}\int _{x_1}^{x_2} \frac{x^2 dx}{ \sqrt{x^2-1}} $$ and then setting $x= ch\psi$, we end up with the formula 
$$ T = \frac{b}{C\sqrt{C}} \left[ x \sqrt{x^2-1} + \log (x+\sqrt{x^2-1)}\right] _{\sqrt{\frac{C}{b}}\sqrt{L-U}} ^ {\sqrt{\frac{C}{b}}\sqrt{L}} $$ 
Assuming $b = 2GM^* \sim GM = 2.10^{37} $, and using the values of the appendix, an approximate value for $T$ is $4.9. 10^{18} s$, about 3/2 times larger than the collapse time when the atttractor is acting alone, cf. \eqref{Tatt}. This suggests that the action of the repeller to reduce the collapse time {\it starting from now} is smaller than that of the attractor, which is consistent with the numerical result \eqref{tot}. But the most important contribution of the repeller has been for $t<0$, when we were closer to the repeller than to the attractor.

\section{Looking for the past action of the repeller} Throughout the text, until now the origin of time was the ``present", more precisely the time when the velocity of the local group with respect to the attractor was  measured. It is natural to ask what happened before to produce a velocity vector pointing towards the attractor. With the equation at hand, this means taking the backward equation and see what happens for negative times. It turns out, as expected, that the solution could not start from the repeller, since in terms of the backward equation \begin{equation} \label{rep} v'=  \sqrt{ \frac{a}{v} - \frac{b}{L-v}+ C }\end{equation} satisfied by $v(t) = u(-t)$ this would imply crossing the equilibrium point, solution of the equation $$ \frac{a}{v} - \frac{b}{L-v}+ C = 0 $$ Since the function $$ \Phi(v) = \frac{a}{v} - \frac{b}{L-v}$$ is decreasing on $(0, L) $ and $$\lim _{v\to 0} \Phi(v) = +\infty;\quad \lim _{v\to L} \Phi(v) = -\infty $$ the equilibrium exists and is unique in $(0, L) $. It can easily be computed by selecting the relevant solution of the equation
$$ a(L-v) -bv + Cv(L-v)= 0 . $$
We skip the details. It is not very interesting to pursue this further since the actual trajectory is  not really supported by the straight line joining the repeller and the attractor, even if it tends to do so in large time by the joint action of both masses. The above calculation suggests that for any initial state (position + velocity), there is a kind of forbidden zone around the repeller, depending only on the total energy and therefore accessible to measurement, that the past trajectory never crossed. 

\section{A problem concerning the repeller} For the real system called ``the Dipole" in the literature,  there is a big mystery about the repeller. If the negative mass is virtual and just corresponds to a density defect compared to the average density of the universe, the equivalent mass will never be comparable to that of the Shapley attractor. Indeed, the average density is about $10^{-28} kg/m^3$. To compete with  the mass of Shapley, we need a volume of $10^{(47+28)} m^3 = 10^{75} m^3 $, corresponding to a diameter greater than $10^{25}$meters, hence $10^9 ly,$ clearly incompatible with the figures.  Negative masses, whichever it may mean, have been considered in so-called bimetric models, cf.\cite{Petit}. Such models question basically our understanding of time, but after all the same was true for the Big Bang model, the theory of relativity and quantum mechanics. \begin{rmk} As a matter of fact, there is no need for the equivalent mass of the repeller to be comparable to that of the attractor in absolute value to explain what happens. If we want to rule out the idea of real negative mass and stick with that of the local density defect, we may imagine for instance a mass $M^*\sim 10^{-3} M$, corresponding to a diameter of about $10^8 ly$ .  In this case the collapse time {\bf starting from now} will be almost identical to the result of \eqref{Tatt}. \end{rmk} 

\section{Appendix} In this section, we indicate the approximate values of the quantities used throughout the paper. 
\subsection {Classical values } 
- Light velocity in m/s: $ c = 299,792,458 m/s \sim 3. 10^5 m/s.$ \newline
- The mass of the sun $ M_S\sim 1.9891 × 10^{30}  kg \sim 2. 10^{30} kg. $  \newline
- The light year in meters:  $ 1 ly \sim 9,461. 10^{15} m. $  \newline
- The gravitational constant $ G\sim 6.7. 10^{-11}  m^3 s^{-2} kg^{-1}$

\subsection {The Shapley attractor and the local group}  

- Mass of the Shapley attractor  $ M\sim 8.10^{16} M_S. \sim 1,6. 10^{47} kg,$\newline
hence $$ K = 2GM \sim 2. 10^{37} m^3 s^{-2}.$$
- The distance of the local group to the Shapley attractor: $ W = 650 Mly \sim 6.15. 10^{24} m. $ \newline
- The distance of the attractor and repeller centers is  $ L \sim 1,4. 10^{25} m. $ \newline
- The velocity $$ |w'(0)| = 6,3. 10^5 m/s \sim 2,1. 10^{-3} c. $$


\begin{thebibliography}{99}


	 \bibitem{Hoff} \textsc {Y. Hoffman, D. Pomarède, R. Brent Tully and H. M. Courtois}, The Dipole Repeller, \emph{Nature Astronomy} 1, 30 (2017).
	 
	 \bibitem{Petit} \textsc {J.-P. Petit and G. d'Agostini},  Cosmological bimetric model with interacting positive and negative masses and two different speeds of light, in agreement with the observed acceleration of the Universe, \emph{Modern Physics Letters} A 29, 34  (2014), 1450182. 
	 
	 \bibitem{Proust}\textsc {D. Proust, H.Quintana, E.R. Carrasco, A. Reisenegger, E. Slezak, H. Muriel, R.Dünner, L. Sodré, Jr., M.J.Drinkwater, Q.A. Parker, , C.J.  Ragone},  The Shapley Supercluster: the Largest Matter Concentration in the Local Universe, \emph{The Messenger}  124 (2006), page 30.
	 	 
	 
	 	 
		\end{thebibliography}
\end{document}